\documentclass[10pt]{amsart} %

\usepackage{float}

\usepackage{epsfig}
\usepackage{epstopdf}
\usepackage{tikz}
\usepackage{array,amsmath, enumerate,  url, psfrag}
\usepackage{amssymb, amsaddr, fullpage}
\usepackage{graphicx,subfigure}
\usepackage{color}

\newtheorem{theorem}{Theorem}[section]
\newtheorem{lemma}[theorem]{Lemma}
\newtheorem{conj}[theorem]{Conjecture}
\newtheorem{proposition}[theorem]{Proposition}
\newtheorem{corollary}[theorem]{Corollary}

\title{Packing $(1,1,2,2)$-coloring of some subcubic graphs}

\author{Runrun Liu$^{1}$, Xujun Liu$^{2}$,  Martin Rolek$^{3}$,    Gexin Yu$^{3}$}

\address{
$^{1}$\small School of Mathematics and Statistics, Central China Normal University, Wuhan, Hubei, China.\\
$^2$\small Department of Mathematics, University of Illinois, Urbana, IL, 61801, USA\\
$^3$\small Department of Mathematics, William \& Mary, Williamsburg, VA, 23185, USA.
}

\thanks{The work is done while the first author was at William \& Mary as a visiting student,  supported by the Chinese Scholarship Council.  The work of the second author is supported by the Waldemar J., Barbara G., and Juliette Alexandra Trjitzinsky Fellowship. The research of the last author was supported in part by the Natural Science Foundation of China (11728102).}

\email{827261672@qq.com (R. Liu), xliu150@illinois.edu, msrolek@wm.edu, gyu@wm.edu}

\begin{document}
\maketitle

\begin{abstract}
For a sequence of non-decreasing positive integers $S = (s_1, \ldots, s_k)$, a packing $S$-coloring is a partition of $V(G)$ into sets $V_1, \ldots, V_k$ such that for each $1\leq i \leq k$ the distance between any two distinct $x,y\in V_i$ is at least $s_i+1$.
The smallest $k$ such that $G$ has a packing $(1,2, \ldots, k)$-coloring is called the packing chromatic number of $G$ and is denoted by $\chi_p(G)$. For a graph $G$, let $D(G)$ denote the graph obtained from $G$ by subdividing every edge. The question whether $\chi_p(D(G)) \le 5$ for all subcubic graphs was first asked by Gastineau and Togni and later conjectured by Bre\v sar, Klav\v zar, Rall and Wash. Gastineau and Togni observed that if one can prove every subcubic graph except the Petersen graph is packing $(1,1,2,2)$-colorable then the conjecture holds. The maximum average degree, mad($G$), is defined to be $\max\{\frac{2|E(H)|}{|V(H)|}: H \subset G\}$. In this paper, we prove that subcubic graphs with $mad(G)<\frac{30}{11}$ are packing $(1,1,2,2)$-colorable. As a corollary, the conjecture of Bre\v sar et al holds for every subcubic graph $G$ with $mad(G)<\frac{30}{11}$.

\end{abstract}

\section{Introduction}	
For a sequence of non-decreasing positive integers $S = (s_1, \ldots, s_k)$, a {\em packing $S$-coloring} of a graph $G$ is a partition of $V(G)$ into sets $V_1, \ldots, V_k$ such that for each $1\leq i \leq k$ the distance between any two distinct $x,y\in V_i$ is at least $s_i+1$.
The smallest $k$ such that $G$ has a packing $(1,2, \ldots, k)$-coloring (packing $k$-coloring) is called {\em the packing chromatic number} of $G$ and is denoted by $\chi_p(G)$.

The notion of packing $k$-coloring was introduced in 2008 by
Goddard,  Hedetniemi, Hedetniemi,  Harris and  Rall~\cite{GHHHR}
motivated by frequency assignment problems in broadcast networks.
There are more than 30 papers on the topic
(e.g.~\cite{ANT1,BF,BKR1,BKR2,BKRW1,BKRW2,BKRW3,CJ1, FG1,FKL1,G1,GT1, GHT, KV1, LBS2, S1, TV1} and references in them). In particular, Fiala and Golovach~\cite{FG1} proved that finding the packing chromatic number of a graph is NP-complete even in the class of trees. Sloper~\cite{S1} showed that the infinite complete ternary tree (every vertex has $3$ child vertices) has unbounded packing chromatic number.

For a graph $G$, let $D(G)$ denote the graph obtained from $G$ by subdividing every edge. The questions on how large can $\chi_p(G)$ and $\chi_p(D(G))$ be if $G$ is a subcubic graph (i.e., a graph with maximum degree at most $3$) were discussed in several papers (\cite{BKRW1,BKRW2,GT1,LBS1,S1}).
In particular, Gastineau and Togni~\cite{GT1} asked whether $\chi_p(D(G))\leq 5$ for every subcubic graph $G$ and
Bre\v sar, Klav\v zar, Rall, and Wash~\cite{BKRW2} later conjectured this.

\begin{conj}[Bre\v sar, Klav\v zar, Rall, and Wash~\cite{BKRW2}]\label{conjbkrw}
Let $G$ be a subcubic graph. Then $\chi_p(D(G))\leq 5$.
\end{conj}

Recently, Balogh, Kostochka and Liu~\cite{BKL} showed that $\chi_p(G)$ is not bounded in the class of cubic graphs. They actually proved a stronger result: for each fixed integer $k \ge 12$ and $g \ge 2k + 2$, almost every $n$-vertex cubic graph of girth at least $g$ has the packing chromatic number greater than $k$. Bre\v sar and Ferme~\cite{BF} later provided an explicit family of subcubic graphs with unbounded packing chromatic number.
In contrast, Balogh, Kostochka and Liu~\cite{BKL2} showed $\chi_p(D(G))$ is bounded by $8$ in the class of subcubic graphs.

The following observation of Gastineau and Togni~\cite{GT1} implies that if one can prove every subcubic graph except the Petersen graph is packing $(1,1,2,2)$-colorable then $\chi_p (D(G)) \le 5$ for every subcubic graph. They also asked a stronger question that whether every subcubic graph except the Petersen graph is packing $(1,1,2,3)$-colorable.

\begin{proposition}[\cite{GT1} Proposition 1]\label{extension}
Let $G$ be a graph and $S
=(s_1,...,s_k)$ be a non-decreasing sequence of integers. If $G$ is $S$-colorable then $D(G)$ is $(1, 2s_1+1, \ldots, 2s_k+1)$-colorable.
\end{proposition}

The problem whether every subcubic graph except the Petersen graph has a packing $(1,1,2,2)$-coloring is very intriguing by itself but only a few subclasses of subcubic graphs were shown to have such a coloring. In particular, Bre\v sar, Klav\v zar, Rall, and Wash~\cite{BKRW2} showed that if $G$ is a generalized prism of a cycle, then $G$ is packing $(1,1,2,2)$-colorable if and only if $G$ is not the Petersen graph. Many similar colorings have also been considered (e.g.~\cite{BKL2, BGT, GT1, GKT, GX1, GX2}). In particular, Gastineau and Togni~\cite{GT1} showed subcubic graphs are packing $(1,2,2,2,2,2,2)$-colorable and packing $(1,1,2,2,2)$-colorable. Balogh, Kostochka and Liu~\cite{BKL2} proved that subcubic graphs are packing $(1,1,2,2,3,3,4)$-colorable with color $4$ used at most once and $2$-degenerate subcubic graphs are packing $(1,1,2,2,3,3)$-colorable. Moreover, Borodin and Ivanova~\cite{BI} proved that every subcubic planar graph with girth at least $23$ has a packing $(2,2,2,2)$-coloring. Bre\v sar, Gastineau and Togni~\cite{BGT} proved very recently that every subcubic outerplanar graph has a packing $(1,2,2,2)$-coloring and their result is sharp in the sense that there exist subcubic outerplanar graphs that are not packing $(1,2,2,3)$-colorable.


In this paper, we consider packing $(1,1,2,2)$-coloring of subcubic graphs  with bounded  maximum average degree, {\em{mad($G$)}}, which is defined to be $\max\{\frac{2|E(H)|}{|V(H)|}: H \subset G\}$.


\begin{theorem}\label{mad2.75}
Every subcubic graph $G$ with $mad(G)<\frac{30}{11}$ is packing $(1,1,2,2)$-colorable.
\end{theorem}

Since planar graphs with girth at least $g$ have maximum average degree less than $\frac{2g}{g-2}$, we obtain the following corollary,  which extends the result of Borodin and Ivanova~\cite{BI} on packing $(1,1,2,2)$-coloring of subcubic planar graphs.

\begin{corollary}\label{girth8}
Every subcubic planar graph with girth at least $8$ is packing $(1,1,2,2)$-colorable.
\end{corollary}

By Proposition~\ref{extension}, we also have the following immediate corollary, which confirms Conjecture~\ref{conjbkrw} for subcubic graphs with maximum average degree less than $\frac{30}{11}$.

\begin{corollary}
If $G$ is a subcubic graph with $mad(G)<\frac{30}{11}$, then $\chi_p(D(G))\le 5$.
\end{corollary}

\begin{proof}
Proposition~\ref{extension} implies that if $G$ is packing $(1,1,2,2)$-colorable then $D(G)$ is packing $(1,3,3,5,5)$-colorable, which implies a packing $(1,2,3,4,5)$-coloring of $D(G)$ and thus $\chi_p(D(G)) \le 5$.
\end{proof}



In the end of this section, we introduce some notations used in the paper. A $k$-vertex ($k^+$-vertex,  $k^-$-vertex) is a vertex of degree $k$ (at least $k$, at most $k$). 
For each $u\in V(G)$, call $v$ a $k$-neighbor of $u$ if $v$ is a neighbor of $u$ and has degree $k$. $N_G^d(u)$ denotes the set of all vertices that are at distance $d$ from $u$.





\section{Proof of Theorem~\ref{mad2.75}}
Let $G$ be a minimum counterexample to Theorem~\ref{mad2.75} with fewest number of vertices. 
For simplicity, we use $(1,1,2,2)$-coloring instead of packing $(1,1,2,2)$-coloring in the rest of this paper.  We assume that the colors are $\{1_a, 1_b, 2_a, 2_b\}$ such that vertices with color $1_a$ (or $1_b$) are not adjacent and vertices with color $2_a$ (or $2_b$) must have distance at least two.


\begin{lemma}\label{minimum}
$\delta(G)\ge2$.
\end{lemma}
\begin{proof}
Suppose otherwise that $v$ is a $1$-vertex in $G$ with $uv\in E(G)$. By the minimality of $G$, $G\setminus \{v\}$ has a $(1,1,2,2)$-coloring $f$. Then we can extend $f$ to $G$ by coloring $v$ with a color in  $\{1_a,1_b\}\setminus \{f(u)\}$, which contradicts the assumption that $G$ is a minimum counterexample.
\end{proof}

\begin{lemma}\label{2_adj}
There are no adjacent $2$-vertices in $G$.
\end{lemma}

\begin{proof}
Suppose otherwise that $u,v$ are adjacent $2$-vertices in $G$. Let $N_G^1(u)=\{u',v\}$ and $N_G^1(v)=\{u,v'\}$. By the minimality of $G$, $G\setminus \{u,v\}$ has a $(1,1,2,2)$-coloring $f$. We color $u$ (respectively $v$) with a color in $\{1_a,1_b\}\setminus \{f(u')\}$ (respectively $\{1_a,1_b\}\setminus \{f(v')\}$). We obtain a $(1,1,2,2)$-coloring of $G$ unless $u,v$ receive the same color. Thus, we may assume $f(u') = f(v') = 1_b$ and $f(u) = f(v) = 1_a$. Moreover, we may assume $d(u')=3$ and $f(u') = \{1_a, 2_a, 2_b\}$, since otherwise we recolor $u$ with a color $x \in \{2_a, 2_b\}\setminus f(u')$ and obtain a $(1,1,2,2)$-coloring of $G$. We obtain a $(1,1,2,2)$-coloring of $G$ by recoloring $u'$ with $1_a$ and $u$ with $1_b$, which is a contradiction.
\end{proof}

We will use lemma~\ref{tool1} extensively in the rest of the paper.

\begin{lemma}\label{tool1}
Let $v$ be a $2$-vertex in $G$ with two neighbors $u,w$. Let $N_G^1(u)=\{v,u_1,u_2\}$ and $N_G^1(w)=\{v,w_1,w_2\}$. Let $f$ be a $(1,1,2,2)$-coloring of $G-v$.  Then either $\{f(u), f(w)\}=\{1_a, 1_b\}$, and $\{1_a,1_b\} \subseteq \{f(u), f(u_1), f(u_2)\}$, $\{1_a,1_b\} \subseteq \{f(w), f(w_1), f(w_2)\}$ and $\{2_a, 2_b\} \subseteq f(N_G^2(v))$; or $f(u)=f(w) \in \{2_a, 2_b\}$, and $\{f(u_1),f(u_2)\}=\{f(w_1),f(w_2)\}=\{1_a,1_b\}$.
\end{lemma}

\begin{proof}
We may color $v$ with some $x \in \{1_a,1_b\}\setminus\{f(u),f(w)\}$ to obtain a $(1,1,2,2)$-coloring of $G$,  unless $ \{f(u), f(w)\} = \{1_a,1_b\}$ or $f(u)=f(w) \in \{2_a, 2_b\}$.

\textbf{Case 1:} $ \{f(u), f(w)\} = \{1_a,1_b\}$. By symmetry, we assume $f(u)=1_a$ and $f(w)=1_b$. We have $1_b \in \{f(u_1), f(u_2)\}$ since otherwise we can recolor $u$ with $1_b$ and color $v$ with $1_a$ to obtain a $(1,1,2,2)$-coloring of $G$. Similarly, we have $1_a \in \{f(w_1), f(w_2)\}$. Moreover, if $\{2_a, 2_b\} \nsubseteq f(N_G^2(v))$ then we can color $v$ with a color $x \in f(N_G^2(v)) \setminus \{2_a, 2_b\}$ to obtain a $(1,1,2,2)$-coloring of $G$. Thus, $\{2_a, 2_b\} \subseteq f(N_G^2(v))$. 

\textbf{Case 2:} $f(u)=f(w) \in \{2_a, 2_b\}$. If $\{f(u_1), f(u_2)\} \neq \{1_a, 1_b\}$, then we recolor $u$ with some $x \in \{1_a, 1_b\}\setminus \{f(u_1), f(u_2)\}$ and color $v$ with $y \in \{1_a, 1_b\} \setminus \{x\}$ to obtain a $(1,1,2,2)$-coloring of $G$. Thus, we have $\{f(u_1),f(u_2)\}=\{1_a,1_b\}$ and similarly $\{f(w_1),f(w_2)\}=\{1_a,1_b\}$.
\end{proof}

By symmetry, whenever the situation in Lemma~\ref{tool1} happens, we may assume $f(u) = 1_a$, $f(w) = 1_b$, $\{f(w_1), f(w_2)\} = \{1_a, 2_a\}$ and $\{f(u_1), f(u_2)\} = \{1_b, 2_b\}$ in the former case and $f(u) = f(w) = 2_a$ in the latter case.

\begin{lemma}\label{2-neighbor}
Each $3$-vertex in $G$ has at most one $2$-neighbor.
\end{lemma}

\begin{proof}
Suppose not, i.e., $u_2$ is a $3$-vertex in $G$ with $N_G^1(u_2)=\{u_1,v_2,u_3\}$ and $d(u_1)=d(u_3)=2$. Let $v_i$ be the neighbors of $u_i$ distinct from $u_2$ for each $i\in\{1,3\}$.  For each $i\in[3]$, let $N_G^1(v_i)=\{u_i,v_i'\}$ if $d(v_i)=2$ and $N_G^1(v_i)=\{u_i,v_i',v_i''\}$ if $d(v_i)=3$. By Lemma~\ref{tool1}, $G-u_1$ has a $(1,1,2,2)$-coloring $f$ such that either $f(v_1)=1_a, f(u_2)=1_b$ or $f(v_1)=f(u_2)=2_a$.

\textbf{Case 1:} $f(v_1)=1_a, f(u_2)=1_b$. By symmetry, we have $\{f(v_1'),f(v_1'')\}=\{1_b,2_b\}$ and $\{f(v_2),f(u_3)\}=\{1_a,2_a\}$.

\textbf{Case 1.1:} $f(v_2)=1_a$ and $f(u_3)=2_a$. If $f(v_3) \neq 1_a$, then we can recolor $u_3$ with $1_a$ and color $u_1$ with $2_a$ to obtain a $(1,1,2,2)$-coloring of $G$, which is a contradiction. Thus, $f(v_3) = 1_a$ and we recolor $u_3$ with $1_b$. If $1_b \notin \{f(v_2'), f(v_2'')\}$, then we recolor $v_2$ with $1_b$, $u_2$ with $1_a$ and color $u_1$ with $1_b$ to obtain a $(1,1,2,2)$-coloring of $G$. Thus, $1_b \in \{f(v_2'), f(v_2'')\}$. If $\{2_a, 2_b\} \nsubseteq \{f(v_2'), f(v_2'')\}$, then we obtain a $(1,1,2,2)$-coloring of $G$ by recoloring $u_2$ with a color $x \in \{2_a, 2_b\} \setminus \{f(v_2'), f(v_2'')\}$ and coloring $u_1$ with $1_b$. Thus, $\{1_b,2_a,2_b\} \subseteq \{f(v_2'),f(v_2'')\}$, which is a contradiction.

\textbf{Case 1.2:} $f(v_2)=2_a$ and $f(u_3)=1_a$. If $f(v_3) \neq 1_b$, then we can recolor $u_3$ with $1_b$, $u_2$ with $1_a$ and color $u_1$ with $1_b$ to obtain a $(1,1,2,2)$-coloring of $G$, which is a contradiction. Thus, $f(v_3) = 1_b$. If $\{1_a, 1_b\} \nsubseteq \{f(v_2'), f(v_2'')\}$, then we obtain a $(1,1,2,2)$-coloring of $G$ by recoloring $v_2$ with a color $x \in \{1_a, 1_b\} \setminus \{f(v_2'), f(v_2'')\}$, $u_2$ with $2_a$ and color $u_1$ with $1_b$. Thus, $\{1_a, 1_b\} \subseteq \{f(v_2'), f(v_2'')\}$. It follows that $2_b \notin \{f(v_2'), f(v_2'')\}$,  and we obtain a $(1,1,2,2)$-coloring of $G$ by recoloring $u_2$ with $2_b$ and coloring $u_1$ with $1_b$, which is a contradiction.

\textbf{Case 2:} $f(v_1)=f(u_2)=2_a$. By symmetry, $f(v_1')=1_a,f(v_1'')=1_b,f(v_2)=1_a, f(u_3)=1_b$. If $f(v_3) \neq 1_a$, then we recolor $u_3$ with $1_a$, $u_2$ with $1_b$ and color $u_1$ with $1_a$. Thus, $f(v_3) = 1_a$. If $1_b \notin \{f(v_3'), f(v_3'')\}$, then we recolor $v_3$ with $1_b$, $u_3$ with $1_a$, $u_2$ with $1_b$ and color $u_1$ with $1_a$. Thus, $1_b \in \{f(v_3'), f(v_3'')\}$. If $\{2_a, 2_b\} \nsubseteq \{f(v_3'), f(v_3'')\}$, then we recolor $u_3$ by a color $x \in \{2_a, 2_b\} \setminus \{f(v_3'), f(v_3'')\}$, $u_2$ with $1_b$ and color $u_1$ with $1_a$ to obtain a $(1,1,2,2)$-coloring of $G$. Therefore, $\{1_b,2_a,2_b\}\subseteq\{f(v_3'),f(v_3'')\}$, which is a contradiction.
\end{proof}

For convenience, call a $3$-vertex $v$ in $G$ {\em special} if all neighbors of $v$ are $3$-vertices.

\begin{lemma}\label{special-3-vertex}
Let $u$ be a $2$-vertex in $G$, then there are at least two special $3$-vertices in $N_G^2(u)$.
\end{lemma}

\begin{proof}
Suppose not, i.e., there are at most one special $3$-vertices in $N_G^2(u)$. Let
$N_G^1(u)=\{u_1,u_2\}$. By Lemma~\ref{2_adj}, both $u_1$ and $u_2$ are $3$-vertices. Let
$N_G^1(u_1)=\{u,v_1,v_2\}$ and $N_G^1(u_2)=\{u,v_3,v_4\}$. By Lemma~\ref{2-neighbor},
$d(v_i)=3$ for each $i \in [4]$ and we may assume by symmetry that both $v_1$ and $v_2$ are
non-special. By Lemma~\ref{2-neighbor} again, $v_1$ (respectively $v_2$) has exactly one $2$-neighbor, say
$w_1$ (respectively $w_3$). Let $N_G^1(v_1)=\{u_1,w_1,w_2\}$, $N_G^1(v_2)=\{u_1,w_3,w_4\}$,
$N_G^1(w_1)=\{v_1,x_1\}$, $N_G^1(w_2)=\{v_1,x_2,x_3\}$, $N_G^1(w_3)=\{v_2,x_4\}$ and $N_G^1(w_4)=\{v_2,x_5,x_6\}$ (note that it is possible that $v_1v_2\in E(G)$). By Lemma~\ref{tool1}, $G-u$ has a $(1,1,2,2)$-coloring $f$ such that either $f(u_1)=1_a,f(u_2)=1_b$ or $f(u_1)=f(u_2)=2_a$.  

\medskip

\textbf{Case 1:} $f(u_1)=1_a$ and $f(u_2)=1_b$. By symmetry, $f(v_1)=1_b,f(v_2)=2_b,f(v_3)=1_a$ and $f(v_4)=2_a$.

\

\textbf{Claim:}
 $\{f(w_1),f(w_2)\}=\{1_a,2_b\}$ and $\{f(w_3),f(w_4)\}=\{1_b,2_a\}$.

\textbf{Proof of Claim:}
If $1_a \notin \{f(w_1), f(w_2)\}$, then we recolor $v_1$ with $1_a$, $u_1$ with $1_b$ and color $u$ with $1_a$ to obtain a $(1,1,2,2)$-coloring of $G$. Thus, $1_a \in \{f(w_1),f(w_2)\}$. If $1_b \notin \{f(w_3), f(w_4)\}$, then we recolor $v_2$ with $1_b$ and color $u$ with $2_b$. Thus, $1_b\in \{f(w_3),f(w_4)\}$. If $2_a \notin f(N_G^2(u_1))$, then we can recolor $u_1$ with $2_a$ and color
$u$ with $1_a$. Thus, $2_a \in \{f(w_1),f(w_2),f(w_3),f(w_4)\}$. Now we may assume that $2_b\notin \{f(w_1),f(w_2)\}$, since otherwise we have $\{f(w_1),f(w_2)\}=\{1_a,2_b\}$ and $ \{f(w_3),f(w_4)\}=\{1_b,2_a\}$ (and we are done). Then $1_a \in \{f(w_3),f(w_4)\}$, since otherwise we can recolor $v_2$ with $1_a$, $u_1$ with $2_b$ and color $u$ with $1_a$. By symmetry, we assume that $f(w_3)=1_a$, $f(w_4)=1_b$ and we also have $\{f(w_1),f(w_2)\}=\{1_a,2_a\}$.

If $f(x_4) \neq 1_b$ or $2_a \notin f(N_G^2(w_3))$, then we recolor $w_3$ with $1_b$ or $2_a$, color $v_2$ with $1_a$, $u_1$ with $2_b$ and $u$ with $1_a$ to obtain a $(1,1,2,2)$-coloring of $G$. Thus, $f(x_4)=1_b$ and $f(N_G^1(x_4)-\{w_3\}) = \{1_a, 2_a\}$, since if $1_a \notin f(N^1_G(x_4) - \{w_3\})$ then we recolor $x_4$ with $1_a$ and it contradicts our previous conclusion that $f(x_4) = 1_b$.

\textbf{Case a:} $f(w_1)=1_a$ and $f(w_2)=2_a$. Then $f(x_1)=1_b$, since otherwise we can recolor $w_1$ with $1_b$, $v_1$ with $1_a$, $u_1$ with $1_b$ and color $u$ with $1_a$ to obtain a $(1,1,2,2)$-coloring of $G$. If $\{1_a,1_b\} \neq \{f(x_2),f(x_3)\}$, then we can recolor $w_2$ with a color $x \in \{1_a, 1_b\}\setminus \{f(x_2),f(x_3)\}$, $v_1$ with $2_a$, $u_1$ with $1_b$ and color $u$ with $1_a$, which is a contradiction. Thus, $ \{f(x_2),f(x_3)\} = \{1_a,1_b\}$. Now we can recolor $v_1$ and $w_3$ with $2_b$, $v_2$ with $1_a$, $u_1$ with $1_b$ and color $u$ with $1_a$, which is a contradiction.

\textbf{Case b:} $f(w_1)=2_a,f(w_2)=1_a$. Then $f(x_1)=1_a$, since otherwise we can recolor $w_1$ with $1_a$, $u_1$ with $2_a$ and color $u$ with $1_a$ to obtain a $(1,1,2,2)$-coloring of $G$. If $1_b \notin \{f(x_2),f(x_3)\}$, then we can recolor $w_2$ with $1_b$, $v_1$ with $1_a$, $u_1$ with $1_b$ and color $u$ with $1_a$. If $2_b \notin \{f(x_2),f(x_3)\}$, then we can recolor $v_1$ and $w_3$ with $2_b$, $v_2$ with $1_a$, $u_1$ with $1_b$ and color $u$ with $1_a$. Thus, we have $\{f(x_2),f(x_3)\}=\{1_b,2_b\}$. Now we can recolor $w_1$ with $1_b$, $v_1$ with $2_a$, $u_1$ with $1_b$ and color $u$ with $1_a$, which is a contradiction. 

This completes the proof of the Claim. 

\

By the Claim, we have the following two subcases.

\textbf{Case 1.1:} $f(w_1)=1_a$ and $f(w_2)=2_b$. Then $f(x_1)=1_b$, since otherwise we can recolor $w_1$ with $1_b$, $v_1$
with $1_a$, $u_1$ with $1_b$ and color $u$ with $1_a$ to obtain a $(1,1,2,2)$-coloring of $G$.
Moreover, $\{f(x_2),f(x_3)\}=\{1_a,1_b\}$, since otherwise we can recolor $w_2$ with $1_a$ or $1_b$, $v_1$ with $2_b$,
$v_2$ with $1_a$, $u_1$ with $1_b$ and color $u$ with $1_a$. Now we can recolor $v_1$ with $2_a$, $u_1$ with $1_b$ and color $u$
with $1_a$, which is a contradiction.

\textbf{Case 1.2:} $f(w_1)=2_b$ and $f(w_2)=1_a$. Then $1_b\in\{f(x_2),f(x_3)\}$, since otherwise we can recolor $w_2$ with $1_b$, $v_1$ with $1_a$, $u_1$ with $1_b$ and color $u$ with $1_a$ to obtain a $(1,1,2,2)$-coloring of $G$. Also $2_b\in\{f(x_2),f(x_3)\}$, since otherwise we can recolor $w_1$ with a color $x \in \{1_a,1_b\}\setminus\{f(x_1)\}$, $v_1$ with $2_b$, $v_2$ with $1_a$, $u_1$ with $1_b$ and color $u$ with $1_a$. Note that $f(x_1)=2_a$, for otherwise we can recolor $v_1$ with $2_a$, $u_1$ with $1_b$
and color $u$ with $1_a$. Now we can recolor $w_1$ and $v_2$ with $1_a$, $u_1$ with $2_b$ and color $u$
with $1_a$, which is a contradiction.

\medskip

\textbf{Case 2:} $f(u_1)=f(u_2)=2_a$. By symmetry, $f(v_1)=1_a,f(v_2)=1_b,f(v_3)=1_a,f(v_4)=1_b$. If $1_b \notin \{f(w_1), f(w_2)\}$, then we recolor $v_1$ with $1_b$, $u_1$ with $1_a$ and color $u$ with $1_b$ to obtain a $(1,1,2,2)$-coloring of $G$. Thus, $1_b\in
\{f(w_1),f(w_2)\}$. Similarly, $1_a\in\{f(w_3),f(w_4)\}$. If $2_b \notin f(N_G^2(u_1))$, then we recolor $u_1$ with $2_b$ and color $u$ with $1_a$. Therefore, $2_b \in\{f(w_1),f(w_2),f(w_3),f(w_4)\}$.

\textbf{Case 2.1:} $f(w_2)=1_b$, $f(w_4)=1_a$.

\textbf{Case 2.1.1:} $f(w_3) \neq 2_b$. Then $f(w_3) = 1_a$ and $f(w_1) = 2_b$. If $2_b\notin\{f(x_2),f(x_3)\}$, then we can recolor $w_1$ with a color $x \in \{1_a,1_b\} \setminus \{f(x_1)\}$, $v_1$ with $2_b$, $u_1$ with $1_a$ and color $u$ with $1_b$. Thus, $2_b \in \{f(x_2),f(x_3)\}$. If $1_a\notin\{f(x_2),f(x_3)\}$, then we can recolor $w_2$ with $1_a$, $v_1$ with $1_b$, $u_1$ with $1_a$ and color $u$ with $1_b$. Therefore, $\{f(x_2),f(x_3)\}=\{1_a, 2_b\}$. Then $f(x_1)=2_a$, for otherwise we can recolor $v_1$ with $2_a$, $u_1$ with $1_a$ and color $u$ with $1_b$. We now recolor $w_1$ with $1_b$. Then we obtain a $(1,1,2,2)$-coloring of $G$ by recoloring $u_1$ with $2_b$ and coloring $u$ with $1_a$ or $1_b$, which is a contradiction.

\textbf{Case 2.1.2:} $f(w_1) \neq 2_b$. Then $f(w_1) = 1_b$ and $f(w_3) = 2_b$. Similarly to Case 2.1.1, we can recolor $w_3$ with $1_a$. Then we obtain a $(1,1,2,2)$-coloring of $G$ by recoloring $u_1$ with $2_b$ and coloring $u$ with $1_a$ or $1_b$, which is a contradiction.

\textbf{Case 2.1.3:} $f(w_1)=f(w_3)=2_b$. Similarly to Case 2.1.1, we can recolor $w_1$ with $1_b$ and $w_3$ with $1_a$. Then we obtain a $(1,1,2,2)$-coloring of $G$ by recoloring $u_1$ with $2_b$ and coloring $u$ with $1_a$ or $1_b$, which is a contradiction.

\textbf{Case 2.2:} $f(w_2) \neq 1_b$ or $f(w_4) \neq 1_a$. By symmetry, we may assume that $f(w_2)=2_b$ and $f(w_1)=1_b$. Then $f(x_1)=1_a$, for otherwise we can recolor $w_1$ with $1_a$, $v_1$ with $1_b$, $u_1$ with $1_a$, and color $u$ with
$1_b$ to obtain a $(1,1,2,2)$-coloring of $G$. If $\{f(x_2), f(x_3)\} \neq \{1_a, 1_b\}$, then $w_2$ can be recolored with $x \in \{1_a, 1_b\} \setminus \{f(x_2), f(x_3)\}$, $v_1$ with $2_b$, $u_1$ with $1_a$, and color $u$ with $1_b$. Therefore, $\{f(x_2),f(x_3)\} = \{1_a, 1_b\}$. We now recolor $v_1$ with
$2_a$, $u_1$ with $1_a$, and color $u$ with $1_b$, which is a contradiction.
\end{proof}


We are now ready to complete the proof of Theorem~\ref{mad2.75}. We use a discharging argument.   Let the initial charge $\mu(v)=d(v) - \frac{30}{11}$ for each $v\in V(G)$.  Since $mad(G)<\frac{30}{11}$, we have
$$\sum\limits_{v \in V(G)} (d(v) - \frac{30}{11}) = 2|E(G)| - n \cdot \frac{30}{11} \le mad(G) \cdot n - \frac{30}{11} \cdot n < 0.$$
To lead to a contradiction, we shall use the following discharging rules to redistribute the charges so that the final charge of every vertex $v$ in $G$, denote by $\mu^*(v)$, is non-negative.
\begin{enumerate}[(R1)]
\item Each special $3$-vertex $v$ gives $\frac{1}{11}$ to each $2$-vertex in $N_G^2(v)$.

\item Each non-special $3$-vertex $v$ gives $\frac{3}{11}$ to each $2$-neighbor.
\end{enumerate}


Let $v$ be a vertex in $G$. By Lemma~\ref{minimum}, $d(v)\in \{2,3\}$. If $d(v)=2$, then by Lemma~\ref{2_adj} and (R2) $v$ gets $\frac{3}{11}$ from each of two $3$-neighbors. By Lemma~\ref{special-3-vertex} there are at least two special $3$-vertices in $N_G^2(v)$ and each of which gives $\frac{1}{11}$ to $v$ by (R1). So $\mu^*(v)\ge 2-\frac{30}{11}+\frac{3}{11}\cdot2+\frac{1}{11}\cdot2=0$. Let $d(v)=3$. If $v$ is not special, then by Lemma~\ref{2-neighbor}, $v$ has exactly one $2$-neighbor, so gives $\frac{3}{11}$ by (R2); if $v$ is special, then $v$ has at most three $2$-vertices in $N_G^2(v)$ by Lemma~\ref{2-neighbor}, so by (R1), $v$ gives $\frac{1}{11}\cdot 3$. So in either case, $\mu^*(v)\ge 3-\frac{30}{11}-\max\{\frac{3}{11},\frac{1}{11}\cdot3\}=0$.


\end{document}